\documentclass [twoside,reqno,  12pt] {amsart}

\usepackage{psfrag}

\usepackage{hyperref}

\usepackage{amsfonts}
\usepackage{amssymb}
\usepackage{a4}
\usepackage{color}

\newtheorem{thm}{Theorem}[section]

\newtheorem{prop}[thm]{Proposition}

\newtheorem{defn}[thm]{Definition}
\newtheorem{rem}[thm]{Remark}

\theoremstyle{definition}

\numberwithin{equation}{section}

\renewcommand{\Re}{\hbox{Re}\,}
\renewcommand{\Im}{\hbox{Im}\,}

\newcommand{\C}{\mathbb{C}}

\newcommand{\R}{\mathbb{R}}

\newcommand{\supp}{\operatorname{supp}}

\def\hat{\widehat}
\def\tilde{\widetilde}
\def \bfo {\begin {eqnarray*} }
\def \efo {\end {eqnarray*} }
\def \ba {\begin {eqnarray*} }
\def \ea {\end {eqnarray*} }
\def \beq {\begin {eqnarray}}
\def \eeq {\end {eqnarray}}
\def \supp {\hbox{supp }}

\def \p {\partial}

\def\hat{\widehat}
\def\tilde{\widetilde}
\def \bfo {\begin {eqnarray*} }
\def \efo {\end {eqnarray*} }
\def \ba {\begin {eqnarray*} }
\def \ea {\end {eqnarray*} }
\def \beq {\begin {eqnarray}}
\def \eeq {\end {eqnarray}}
\def \supp {\hbox{supp }}

\def \p {\partial}

\begin{document}

 \title[Reconstruction on  transversally anisotropic manifolds]{Reconstruction in the Calder\'on problem on conformally transversally anisotropic manifolds}

\author[Feizmohammadi]{Ali Feizmohammadi}

\address
        {Ali Feizmohammadi, Department of Mathematics\\
 University College London, London, \\
 UK-WC1E 6BT, United Kingdom}

\email{a.feizmohammadi@ucl.ac.uk}

\author[Krupchyk]{Katya Krupchyk}

\address
        {K. Krupchyk, Department of Mathematics\\
University of California, Irvine\\
CA 92697-3875, USA }

\email{katya.krupchyk@uci.edu}

\author[Oksanen]{Lauri Oksanen}

\address
        {Lauri Oksanen, Department of Mathematics\\
 University College London, London, \\
 UK-WC1E 6BT, United Kingdom}

\email{l.oksanen@ucl.ac.uk}

\author[Uhlmann]{Gunther Uhlmann}

\address
       {G. Uhlmann, Department of Mathematics\\
       University of Washington\\
       Seattle, WA  98195-4350\\
       USA\\
        and Institute for Advanced Study of the Hong Kong University of Science and Technology}
\email{gunther@math.washington.edu}

\maketitle

\begin{abstract} 
We show that a continuous potential $q$ can be constructively determined  from the knowledge of the Dirichlet--to--Neumann map for the Schr\"odinger operator $-\Delta_g+q$ on a conformally transversally anisotropic manifold of dimension $\geq 3$, provided that the geodesic ray transform on the transversal manifold is constructively invertible. This is a constructive counterpart of the uniqueness result of \cite{DKurylevLS_2016}.  A crucial role in our reconstruction procedure is played by a constructive determination of the boundary traces of suitable complex geometric optics solutions based on Gaussian beams quasimodes concentrated along non-tangential geodesics on the transversal manifold, which enjoy uniqueness properties. This is achieved by applying the simplified version of the approach of \cite{Nachman_Street_2010} to our setting. We also identify the main space introduced in \cite{Nachman_Street_2010} with a standard Sobolev space on the boundary of the manifold.  Another ingredient in the proof of our result is a reconstruction formula for the boundary trace of a continuous potential  from the knowledge of the Dirichlet--to--Neumann map. 
\end{abstract}

\section{Introduction and statement of results}

Let $(M,g)$ be a smooth compact oriented Riemannian manifold of dimension $n\ge 3$ with smooth boundary $\p M$. Let us consider the Dirichlet problem for the Laplace--Beltrami operator $-\Delta_g=-\Delta$, 
\begin{equation}
\label{eq_int_1}
\begin{cases}
-\Delta u=0 & \text{in}\quad M^{\text{int}},\\
u|_{\p M}=f. 
\end{cases}
\end{equation}
Here and in what follows $M^{\text{int}}=M\setminus \p M$.
For any $f\in  H^{\frac{1}{2}}(\p M)$, the problem \eqref{eq_int_1} has a unique solution $u\in H^1(M^{\text{int}})$.  Associated to \eqref{eq_int_1}, we define the Dirichlet--to--Neumann map $\Lambda_{g,0}: H^{\frac{1}{2}}(\p M)\to H^{-\frac{1}{2}}(\p M)$, formally given by $\Lambda_{g,0} f=\p_\nu u|_{\p M}$, where $\p_\nu$ is the unit outer normal to $\p M$. 
If $\psi: M\to M$ is a diffeomorphism satisfying $\psi|_{\p M}=I$ then $\Lambda_{\psi^*g, 0}=\Lambda_{g,0}$, see \cite{Lee_Uhlmann_1989}. 

The anisotropic Calder\'on problem concerns the question of whether the equality $\Lambda_{g_1,0}=\Lambda_{g_2,0}$ implies that $g_2=\psi^* g_1$ where $\psi: M\to M$ is a diffeomorphism such that  $\psi|_{\p M}=I$. This problem is solved for real-analytic metrics in \cite{Lee_Uhlmann_1989},  \cite{Lassas_Uhlmann_2001}, \cite{Lassas_Taylor_Uhlmann_2003}, see also \cite{Guillarmou_Sa_Barreto_2009}, while it remains open in the smooth category, in dimensions $n\ge 3$. The corresponding two dimensional problem, with an additional obstruction arising from the conformal invariance of the Laplacian, is settled in \cite{Lassas_Uhlmann_2001}. 
 
A powerful method for studying the anisotropic Calder\'on problem on genuinely non-analytic manifolds, where the metric is of special form, is introduced in the work \cite{DKSaloU_2009}. The method is based on the technique of Carleman estimates with limiting Carleman weights. The notion of a limiting Carleman weight for the Laplacian was introduced and applied to the Calder\'on problem in the Euclidean setting in \cite{Kenig_Sjostrand_Uhlmann}. An important result of  \cite{DKSaloU_2009} states that on a simply connected open manifold, the existence of a limiting Carleman weight is equivalent to the existence of a parallel unit vector field for a conformal multiple of the metric. Locally, the latter condition is equivalent to the fact that the manifold is conformal to the product of a Euclidean interval and some Riemannian manifold of dimension $n-1$. Following \cite{DKSaloU_2009}, \cite{DKurylevLS_2016},  we have the following definitions.
\begin{defn}
Let $(M,g)$ be a smooth compact oriented Riemannian manifold of dimension $n\ge 3$ with smooth boundary $\p M$.
\begin{itemize}
\item[(i)] $(M,g)$ is called transversally anisotropic if $(M,g)\subset \subset (T^{\text{int}},g)$ where $T=\R\times M_0$, $g=e\oplus g_0$, $(\R,e)$ is the Euclidean line, and $(M_0,g_0)$ is a compact $(n-1)$--dimensional manifold with boundary, called the transversal manifold. 

\item[(ii)] $(M,g)$ is called conformally transversally anisotropic (CTA) if $(M,cg)$ is transversally anisotropic, for some smooth positive function $c$. 

\item[(iii)] $(M,g)$ is called admissible if $(M,g)$ is CTA and the transversal manifold $(M_0,g_0)$ is simple, meaning that 
for any $p\in M_0$, the exponential map $\exp_p$ with its maximal domain of definition in $T_p M_0$ is a diffeomorphism onto $M_0$, and $\p M_0$ is strictly convex.
\end{itemize}
\end{defn}

An interesting special case of the anisotropic Calder\'on problem is such a problem in a fixed conformal class. Since any conformal diffeomorphism fixing the boundary must be the identity map, there is no obstruction to uniqueness arising from isometries in this problem, see \cite{Lionheart_1997}.  The uniqueness for the anisotropic Calder\'on problem in a fixed conformal class was obtained in  \cite{DKSaloU_2009}  in the case of admissible manifolds. Thanks to the simplicity of the transversal manifold, the proof relies on a construction of complex geometric optics solutions by means of a global WKB method, and the injectivity of the attenuated geodesic ray transform on simple manifolds. A reconstruction procedure for the uniqueness result of \cite{DKSaloU_2009} was given in \cite{Kenig_Salo_Uhlmann_2011}. 

Dropping the simplicity assumption on the transversal manifold, the anisotropic Calder\'on problem in a fixed conformal class on a general CTA manifold was studied in \cite{DKurylevLS_2016}. Here the global uniqueness was established under the assumption that the geodesic ray transform on the transversal manifold $(M_0,g_0)$ is injective. In this case, a global WKB approach no longer seems possible, and the complex geometric optics solutions are obtained via a Gaussian beams quasimode construction. We refer to \cite{DKurylevLLimS}, \cite{Krup_Liimatainen_Salo} for the study of the linearized anisotropic Calder\'on problem on transversally anisotropic manifolds.

The goal of this note is to provide a reconstruction procedure for the uniqueness results of \cite{DKurylevLS_2016}. To state our results, let us first give the following definition. 
\begin{defn}
We say that the geodesic ray transform on the transversal manifold $(M_0,g_0)$ is constructively invertible  if any function $f\in C(M_0)$ can be reconstructed from the knowledge of its integrals over all non-tangential geodesics in $M_0$. Here a unit speed geodesic $\gamma:[0,L]\to M_0$ is called non-tangential if $\dot{\gamma}(0), \dot{\gamma}(L)$ are non-tangential vectors on $\p M_0$ and $\gamma(t)\in M_0^{\emph{\text{int}}}$ for all $0<t<L$.  
\end{defn}

Our first result is as follows. 
\begin{thm}
\label{thm_main_1}
Let $(M,g)$ be a given CTA manifold and assume that the geodesic ray transform on the transversal manifold $(M_0,g_0)$ is constructively invertible. If $0< c\in C^\infty(M)$  then from the knowledge of $\Lambda_{cg,0}$
one can constructively determine $c$. 
\end{thm}

An inverse problem closely related to the anisotropic Calder\'on problem is the inverse boundary problem for the Schr\"odinger equation, which we shall proceed to discuss next. Let $q\in C(M)$, and consider Dirichlet problem
\begin{equation}
\label{eq_int_2}
\begin{cases}
(-\Delta+q)u=0 & \text{in}\quad M^{\text{int}},\\
u|_{\p M}=f.
\end{cases}
\end{equation}
In what follows assume  that $0$ is not a Dirichlet eigenvalue of $-\Delta+q$ in $M$ so that $-\Delta+q: (H^2\cap H^1_0)(M)\to L^2(M)$ is bijective.  Under this assumption, for any $f\in H^{\frac{1}{2}}(\p M)$, the problem \eqref{eq_int_2}
has a unique solution $u\in H^1(M^{\text{int}})$. Associated to \eqref{eq_int_2}, we define the Dirichlet--to--Neumann map $\Lambda_{g,q}: H^{\frac{1}{2}}(\p M)\to H^{-\frac{1}{2}}(\p M)$ by 
\begin{equation}
\label{eq_int_2_DN}
\langle \Lambda_{g, q} f, k \rangle_{H^{-\frac{1}{2}}(\p M), H^{\frac{1}{2}}(\p M)}=\int_M( \langle du, dv\rangle_g +qu v)dV_g,
\end{equation}
Here $k\in H^{\frac{1}{2}}(\p M)$, $v\in H^1(M^{\text{int}})$ is such that $v|_{\p M}=k$, and $\langle \cdot, \cdot\rangle_g$ is the pointwise scalar product in the space of $1$-forms. 

Our second result is as follows. 
\begin{thm}
\label{thm_main_2}
Let $(M,g)$ be a given CTA manifold and assume that the geodesic ray transform on the transversal manifold $(M_0,g_0)$ is constructively invertible. Let $q\in C(M)$ be such that $0$ is not a Dirichlet eigenvalue of $-\Delta+q$ in $M$. Then the knowledge of $\Lambda_{g,q}$ determines $q$  constructively. 
\end{thm}

\begin{rem}The constructive invertibility of the geodesic ray transform is known in the following cases, in particular: 
\begin{itemize}
\item $(M_0,g_0)$ is a simple Riemannian surface. In this case, there is a Fredholm type inversion formula established in \cite{Pestov_Uhlmann_2004}, which leads to the exact inversion in the constant curvature case, see \cite{Pestov_Uhlmann_2004},  and its small perturbations, see  \cite{Krishnan_2010}. See also \cite{Monard_2014}, \cite{Monard_2014_2}.

\item  $(M_0,g_0)$ is a Riemannian surface with strictly convex boundary, no conjugate points, and the hyperbolic trapped set (these conditions are satisfied in negative curvature, in particular). In this case, a Fredholm type inversion formula was obtained in  \cite{Guillarmou_Monard_2017}. The inversion formula becomes exact in a neighborhood of a constant negatively curved metric. 

\item $(M_0, g_0)$ is of dimension $n\ge 3$, has a strictly convex boundary and is globally foliated by strictly convex hypersurfaces. In this setting, a layer stripping type algorithm for reconstruction was developed in \cite{Uhlmann_Vasy_2016}. 

\end{itemize}

\end{rem}

\begin{rem}  Theorem \ref{thm_main_1} and Theorem \ref{thm_main_2} are valid in the case of admissible manifolds in particular, thereby providing an alternative proof of the reconstruction results established in \cite{Kenig_Salo_Uhlmann_2011}.  Furthermore, Theorem \ref{thm_main_2} improves the regularity assumption on the potential in the corresponding result of  \cite{Kenig_Salo_Uhlmann_2011} as it holds for a continuous potential whereas the corresponding result of \cite{Kenig_Salo_Uhlmann_2011} requires that the potential should be smooth. 
\end{rem}

Following the pioneering works \cite{Nachman_1988}, \cite{Novikov_1988},  we know that the crucial step in the reconstruction procedure of a potential from the corresponding Dirichlet--to--Neumann map consists of constructively determining the boundary traces of suitable complex geometric optics solutions. To the best of our knowledge, there exist two approaches to the reconstruction of such boundary traces. The first one is due to \cite{Nachman_1988} in the Euclidean setting, where suitable complex geometric optics solutions are constructed globally on all of $\R^n$, enjoying uniqueness properties characterized by decay at infinity. The second one is due to \cite{Nachman_Street_2010}, where the complex geometric optics solutions are constructed by means of Carleman estimates on a bounded domain, and the notion of uniqueness is obtained by restricting the attention to solutions of minimal norm. A common point of both approaches is that the boundary traces of the complex geometric optics solutions in question are determined as unique solutions of well posed integral equations on the boundary of the domain, involving the Dirichlet--to--Neumann map along with other known quantities. The approach of \cite{Nachman_1988}  was extended to the case of admissible manifolds in \cite{Kenig_Salo_Uhlmann_2011} for the Schr\"odinger equation, see also \cite{Campos_2019} for the magnetic Schr\"odinger case in a cylindrical setting. The approach of \cite{Nachman_Street_2010}, which was developed for the Calder\'on problem with partial data, was extended to admissible manifolds in \cite{Assylbekov_2017}, also in the partial data case. 

In this note, we give a simplified presentation of the method of  \cite{Nachman_Street_2010}  in the full data case, for manifolds admitting limiting Carleman weights. We also point out that the space $\mathcal{H}(\p M)$ introduced in \cite{Nachman_Street_2010},  where the main boundary integral equation is solved, agrees with the standard Sobolev space $H^{-\frac{1}{2}}(\p M)$. 

We proceed next to discuss the ideas of the proof of Theorem \ref{thm_main_1} and Theorem \ref{thm_main_2}. First we observe that Theorem \ref{thm_main_1} follows from Theorem \ref{thm_main_2} along exactly the same lines as in  \cite{Kenig_Salo_Uhlmann_2011}, and therefore only Theorem \ref{thm_main_2} will be proved. 
In doing so, we shall rely on complex geometric optics solutions constructed on general CTA manifolds, based on Gaussian beams quasimodes along non-tangential geodesics in $M_0$. Such solutions were constructed in  \cite{DKurylevLS_2016} without any notion of uniqueness involved. In this note, following the method of \cite{Nachman_Street_2010}, we refine this construction somewhat and obtain complex geometric optics solutions enjoying uniqueness properties.  Another ingredient in the proof of Theorem \ref{thm_main_2}  is a reconstruction formula for $q|_{\p M}$  from the knowledge of the Dirichlet--to--Neumann map $\Lambda_{g,q}$ which is performed in Appendix \ref{app_boundary_reconstruction}. This is precisely the result which allows us to improve the regularity of potential in the result of \cite{Kenig_Salo_Uhlmann_2011}.

Finally, we would like to point out that similarly to the reconstructions results of \cite{Kenig_Salo_Uhlmann_2011}, in the reconstruction procedure developed in this note, we make no claims regarding its practicality, our purpose merely being to show that all the steps in the proof of the uniqueness result of \cite{DKurylevLS_2016}  can be carried out constructively. 

The plan of this paper is as follows. Section \ref{sec_Nachman-Street} gives a presentation of the method of  \cite{Nachman_Street_2010}  in the full data case, for manifolds admitting a limiting Carleman weight.  Section \ref{sec_proof_of_thm_main} is devoted to the construction of complex geometric optics solutions enjoying uniqueness properties and to the proof of Theorem \ref{thm_main_2}. Appendix \ref{app_boundary_reconstruction} contains  a reconstruction formula for the boundary traces of a continuous potential from the knowledge of the Dirichlet--to--Neumann map.

\section{The Nachman--Street argument on manifolds admitting a limiting Carleman weight }

\label{sec_Nachman-Street}

The discussion in this section can be regarded as a simplified version of the constructive approach of \cite{Nachman_Street_2010} to determining boundary traces of complex geometric optics solutions, in the full data case, in the setting of compact manifolds with boundary admitting a limiting Carleman weight. We also identify the space $\mathcal{H}(\p M)$ of \cite{Nachman_Street_2010}, where the main boundary integral equation is posed, with the standard Sobolev space $H^{-\frac{1}{2}}(\p M)$. 

Let $(M,g)$ be a smooth compact Riemannian manifold of dimension $n\ge 3$ with smooth boundary $\p M$. Let us consider the semiclassical Laplace--Beltrami operator $-h^2\Delta_g=-h^2\Delta$ on $M$, where $h>0$ is a small semiclassical parameter. Assume, as we may,  that $(M,g)$ is embedded in a compact smooth Riemannian manifold $(N,g)$ without boundary of the same dimension, and let $U$ be open in $N$ such that $M\subset U$. 

 Let $\varphi\in C^\infty(U;\R)$ and let us consider the conjugated operator
 \begin{equation}
\label{eq_Car_-2}
P_\varphi=e^{\frac{\varphi}{h}}(-h^2\Delta)e^{-\frac{\varphi}{h}}=-h^2\Delta -|\nabla \varphi|^2+2\langle \nabla \varphi, h\nabla\rangle +h\Delta\varphi,
\end{equation}
with the semiclassical principal symbol 
\begin{equation}
\label{eq_Car_1}
p_\varphi=|\xi|^2-|d\varphi|^2+2i \langle \xi, d\varphi\rangle\in C^\infty(T^*U).
\end{equation}
Here and in what follows we use $\langle \cdot, \cdot \rangle$ and $|\cdot|$ to denote the Riemannian scalar product and norm both on the tangent and cotangent space.  

Following \cite{Kenig_Sjostrand_Uhlmann},  \cite{DKSaloU_2009}, we say that $\varphi\in C^\infty(U;\R)$ is a limiting Carleman weight for $-h^2\Delta$ on $(U,g)$ if $d\varphi\ne 0$ on $U$, and the Poisson bracket of $\Re p_\varphi$ and $\Im p_\varphi$ satisfies,
\[
\{\Re p_\varphi, \Im p_\varphi\}=0 \quad\text{when}\quad  p_\varphi=0. 
\]
We refer to \cite{DKSaloU_2009} for a characterization of Riemannian manifolds admitting limiting Carleman weights.

Our starting point is the following Carleman estimates for $-h^2\Delta$, established in  \cite{DKSaloU_2009}, see also \cite{Krup_Uhlmann_2018}. 
\begin{prop}
Let $\varphi\in C^\infty(U;\R)$ be a limiting Carleman weight for $-h^2\Delta$ on $(U,g)$. Then for all $0<h\ll 1$, we have 
\begin{equation}
\label{eq_2_1}
h\|u\|_{L^2(M)}\le C\|P_\varphi u\|_{L^2(M)}, \quad C>0, 
\end{equation}
for all $u\in C^\infty_0(M^{\text{int}})$. 
\end{prop}

Note that if $\varphi$ is a limiting Carleman weight for $-h^2\Delta$ then so is $-\varphi$.  Let $P_\varphi^*$ be the formal $L^2(M)$--adjoint of $P_\varphi$.  We have $P_\varphi^*=P_{-\varphi}$. Let us also introduce the following closed subspace of $L^2(M)$,
\[
\text{Ker}(P_\varphi)=\{u\in L^2(M): P_\varphi u=0\}, 
\]
so that we have
\[
L^2(M)=\text{Ker}(P_\varphi)\oplus (\text{Ker}(P_\varphi))^\perp. 
\]

Following \cite{Nachman_Street_2010}, we shall now proceed to construct Green's operator for $P_\varphi$. To that end, we have the following solvability result. 
\begin{prop}
\label{prop_solvab}
Let $\varphi\in C^\infty(U;\R)$ be a limiting Carleman weight for $-h^2\Delta$ on $(U,g)$. Then for all $0<h\ll 1$ and any $v\in L^2(M)$, there is a unique solution $u\in  (\emph{\text{Ker}} (P_\varphi))^\perp $ of the equation
\begin{equation}
\label{eq_2_2}
P_\varphi u=v\quad \text{in}\quad M^{\text{int}}.
\end{equation}
Furthermore, $u$ satisfies the bound
\begin{equation}
\label{eq_2_3}
\|u\|_{L^2(M)}\le \frac{C}{h}\|v\|_{L^2(M)},
\end{equation}
for all $0<h\ll 1$. 
\end{prop}

\begin{proof}
Let $v\in L^2(M)$ and let us first show the existence of a solution to \eqref{eq_2_2} in the space $(\text{Ker}(P_\varphi))^\perp$. To that end, consider the following linear functional
\[
L:P_\varphi^*C_0^\infty(M^{\text{int}})\to \C, \quad L(P_\varphi^* w)=(w,v)_{L^2(M)}.
\]
By the Carleman estimate \eqref{eq_2_1} for $P_\varphi^*$, the map $L$ is well-defined.  Let $w\in C^\infty_0(M^{\text{int}})$. Then by using the Carleman estimate \eqref{eq_2_1} for $P_\varphi^*$ again, we get 
\[
| L(P_\varphi^* w)|\le \frac{C}{h}\| P_\varphi^* w\|_{L^2(M)}\|v\|_{L^2(M)}.
\]
By continuity,  $L$ extends to a linear continuous functional on the closed subspace $\overline{P_\varphi^*C_0^\infty(M^{\text{int}})}\subset L^2(M)$. By the Riesz representation theorem, there is therefore a unique $u\in \overline{P_\varphi^*C_0^\infty(M^{\text{int}})}$ such that 
\[
L(g)=(g,u)_{L^2(M)}, \quad g\in \overline{P_\varphi^*C_0^\infty(M^{\text{int}})},
\]
and $u$ satisfies the bound \eqref{eq_2_3}. In particular, for any $w\in C^\infty_0(M^{\text{int}})$, we have 
\[
L(P_\varphi^* w)=(w,v)_{L^2(M)}= (P_\varphi^* w, u)_{L^2(M)}, 
\]
showing that $P_\varphi u=v$. Furthermore,  we clearly have $(P_\varphi^*C_0^\infty(M^{\text{int}}))^\perp=\text{Ker}(P_\varphi)$, which is equivalent to  
\begin{equation}
\label{eq_2_3_1}
\overline{P_\varphi^*C_0^\infty(M^{\text{int}})}=(\text{Ker}(P_\varphi))^\perp.
\end{equation}
To see the uniqueness, let $u, \tilde u\in  (\text{Ker}(P_\varphi))^\perp $ be solutions to \eqref{eq_2_2}. Then $u-\tilde u\in \text{Ker}(P_\varphi)\cap (\text{Ker}(P_\varphi))^\perp=\{0\}$. 
\end{proof}

Following \cite{Bukhgeim_Uhlmann_2002}, we introduce the Hilbert space 
\[
H_\Delta(M)=\{u\in L^2(M): \Delta u\in L^2(M)\},
\]
the maximal domain of the Laplacian, equipped with the norm
\[
\|u\|_{H_\Delta(M)}^2=\|u\|_{L^2(M)}^2+ \|\Delta u\|_{L^2(M)}^2. 
\]

We have the following result on the existence of  Green's operators, see \cite[Theorem 3.2]{Nachman_Street_2010}.  
\begin{thm}
\label{thm_green_operators}
Let $\varphi\in C^\infty(U;\R)$ be a limiting Carleman weight for $-h^2\Delta$ on $(U,g)$. Then for all $0<h\ll 1$, there exists a linear continuous operator $G_{\varphi}:L^2(M)\to L^2(M)$ such that 
\begin{itemize}
\item[(i)] $P_{\varphi} G_{ \varphi}=I$ on $L^2(M)$, 
\item[(ii)] $\|G_{ \varphi}\|_{\mathcal{L}(L^2(M), L^2(M))}=\mathcal{O}(h^{-1})$, 
\item[(iii)] $G_{ \varphi}:L^2(M)\to e^{\frac{\varphi}{h}}H_{\Delta}(M)$,
\item[(iv)] $G_\varphi^*=G_{-\varphi}$,
\item[(v)] $G_{ \varphi} P_{\varphi}=I$ on $C^\infty_0(M^{\text{int}})$. 
\end{itemize}
\end{thm} 

\begin{proof}

We follow \cite{Nachman_Street_2010} and present the proof for the completeness and convenience of the reader. First we define the following solution operator to the equation \eqref{eq_2_2}, 
\[
H_\varphi:L^2(M)\to \overline{P_\varphi^*C_0^\infty(M^{\text{int}})}=(\text{Ker}(P_\varphi))^\perp, \quad H_\varphi(v)=u,
\]
where $u\in \overline{P_\varphi^*C_0^\infty(M^{\text{int}})}$ is the unique solution to the equation \eqref{eq_2_2}, see Proposition \ref{prop_solvab}. It is clear that $H_{ \varphi}$ satisfies the properties (i)--(iii) of Theorem \ref{thm_green_operators}. However, to achieve the property (iv) we shall make a suitable modification of $H_{\varphi}$. 

In doing so, let $1-\pi_{\varphi}$ be the orthogonal projection onto $\overline{P_\varphi^*C_0^\infty(M^{\text{int}})}\subset L^2(M)$. Then we claim that $\pi_\varphi$ is the orthogonal projection onto $\text{Ker}(P_\varphi)$.  Indeed, we have 
\[
\text{Ran}(\pi_\varphi)=(\text{Ker}(\pi_\varphi))^\perp,
\]
and therefore, it suffices to show that 
\[
\text{Ker}(\pi_\varphi)=(\text{Ker}(P_\varphi))^\perp.
\]
This follows from the fact that 
\[
\text{Ker}(\pi_\varphi)=\text{Ran}(1-\pi_\varphi)=\overline{P_\varphi^*C_0^\infty(M^{\text{int}})}
\]
and \eqref{eq_2_3_1}. This completes the proof of the claim. 

Let 
\begin{equation}
\label{eq_2_4_-1}
T_\varphi=H_\varphi(1-\pi_{-\varphi}), \quad T_{-\varphi}=H_{-\varphi}(1-\pi_{\varphi}).
\end{equation}
We claim that 
\begin{equation}
\label{eq_2_4}
T_\varphi^*=T_{-\varphi}. 
\end{equation}
Indeed, first, we have 
\[
T_{-\varphi}\pi_\varphi=0, \quad 
T^*_{\varphi}\pi_\varphi=(1-\pi_{-\varphi})(\pi_\varphi H_\varphi)^*=0.
\]
As $\text{Ran}(1-\pi_\varphi)=\overline{P_\varphi^*C_0^\infty(M^{\text{int}})}$, to prove \eqref{eq_2_4} it suffices to show that 
\begin{equation}
\label{eq_2_5}
T^*_\varphi P^*_\varphi w=T_{-\varphi}P^*_\varphi w = H_{-\varphi}P^*_\varphi w, 
\end{equation}
for all $w\in C^\infty_0(M^{\text{int}})$. To see \eqref{eq_2_5}, we first observe that for any $x\in L^2(M)$, 
\[
(T^*_\varphi P^*_\varphi w, \pi_{-\varphi}x)_{L^2(M)}= (H^*_\varphi P^*_\varphi w, (1-\pi_{-\varphi})\pi_{-\varphi}x)_{L^2(M)}=0,
\]
and 
\[
(H_{-\varphi}P^*_\varphi w,  \pi_{-\varphi}x)_{L^2(M)}=(\pi_{-\varphi} H_{-\varphi}P^*_\varphi w,  x)_{L^2(M)}=0.
\]
As $\text{Ran}(1-\pi_{-\varphi})=\overline{P_{\varphi}C_0^\infty(M^{\text{int}})}$, to complete the  proof of \eqref{eq_2_5} it remains to check that 
\begin{equation}
\label{eq_2_6}
(T^*_\varphi P^*_\varphi w, P_\varphi g)_{L^2(M)}=(H_{-\varphi}P^*_\varphi w, P_\varphi g)_{L^2(M)}
\end{equation}
for all $g\in C^\infty_0(M^{\text{int}})$.  Integrating by parts, and using that $P_{-\varphi}H_{-\varphi}=1$ on $L^2(M)$,  we get   
\begin{equation}
\label{eq_2_7}
(H_{-\varphi}P^*_\varphi w, P_\varphi g)_{L^2(M)}= (P_{-\varphi} H_{-\varphi}P^*_\varphi w,  g)_{L^2(M)}=(P^*_\varphi w,  g)_{L^2(M)}.
\end{equation}
On the other hand, we obtain that 
\begin{equation}
\label{eq_2_8}
\begin{aligned}
(T^*_\varphi P^*_\varphi w, P_\varphi g)_{L^2(M)}= (H_\varphi^*P^*_\varphi w, (1-\pi_{-\varphi})P_\varphi g)_{L^2(M)}\\
=(w, P_\varphi H_\varphi P_\varphi g)_{L^2(M)}= (P^*_\varphi w,g)_{L^2(M)}.
\end{aligned}
\end{equation}
It follows from \eqref{eq_2_7} and \eqref{eq_2_8} that \eqref{eq_2_6} holds. This complete the proof of \eqref{eq_2_4}.

Finally, let us define 
\[
G_\varphi=H_\varphi+\pi_\varphi H^*_{-\varphi}, \quad G_{-\varphi}=H_{-\varphi}+\pi_{-\varphi}H^*_{\varphi}. 
\]
It is clear that $G_{\pm \varphi}$ satisfies all the properties (i)--(iii) in Theorem \ref{thm_green_operators}. To see the property (iv), using \eqref{eq_2_4_-1} and \eqref{eq_2_4}, we get 
\[
G_\varphi^*=H_\varphi^*+H_{-\varphi}\pi_\varphi= (T_\varphi+H_\varphi\pi_{-\varphi})^*+H_{-\varphi}-T_{-\varphi}=G_{-\varphi}.
\]
To see the property (v), letting  $f,g\in C^\infty_0(M^{\text{int}})$, we obtain that 
\[
(G_\varphi P_\varphi f,g)_{L^2(M)}=(P_\varphi f, G_{-\varphi} g)_{L^2(M)}=( f, P_{-\varphi}G_{-\varphi} g)_{L^2(M)}=(f,g)_{L^2(M)}.
\]
This completes the proof of Theorem \ref{thm_green_operators}.
\end{proof}

Let $\gamma: C^\infty(M)\to C^\infty(\p M)$, $\gamma(u)=u|_{\p M}$ be the trace map. It is shown in \cite{Bukhgeim_Uhlmann_2002},  see also \cite[Section 26.2]{Eskin_book}, that the trace map $\gamma$ extends to a continuous map 
\begin{equation}
\label{eq_gamma}
\gamma: H_\Delta(M)\to H^{-\frac{1}{2}}(\p M).
\end{equation}
We claim that the map $\gamma$ in \eqref{eq_gamma} is surjective. This follows from the fact that if $g\in H^{-\frac{1}{2}}(\p M)$, there exists a unique $u\in L^2(M)$ such that $-\Delta u=0$ in $M$ and $\gamma(u)=g$, see \cite[Theorem 26.3]{Eskin_book}. Furthermore, we have 
\begin{equation}
\label{eq_ref_Eskin_1}
\|u\|_{L^2(M)}\le C\|g\|_{H^{-\frac{1}{2}}(\p M)},
\end{equation}
see \cite[Theorem 26.3]{Eskin_book}.

\textbf{Remark}.  \emph{The space 
\[
\mathcal{H}(\p M)=\{\gamma(u): u\in H_\Delta(M)\}\subset H^{-\frac{1}{2}}(\p M)
\]
was introduced in \cite{Nachman_Street_2010}. The discussion above shows that in fact we have 
\[
\mathcal{H}(\p M)=H^{-\frac{1}{2}}(\p M).
\]}

Let $q\in L^\infty(M)$ and let us assume from here on that $0$ is not a Dirichlet eigenvalue of $-\Delta+q$, so that 
$-\Delta+q:(H^2\cap H^1_0)(M^{\text{int}})\to L^2(M)$ 
is bijective. 
Setting
\[
b_q:=\{u\in L^2(M): (-\Delta+q)u=0\}\subset H_\Delta(M), 
\]
we claim that the trace map 
\[
\gamma:b_q\to H^{-\frac{1}{2}}(\p M)
\]
 is bijective. 
Indeed, to see the surjectivity of $\gamma$, let $g\in H^{-\frac{1}{2}}(\p M)$. By \cite[Theorem 26.3]{Eskin_book},  there exists  $u\in L^2(M)$ such that $-\Delta u=0$ in $M^{\text{int}}$ and $\gamma(u)=g$.  As $0$ is not a Dirichlet eigenvalue of $-\Delta+q$, the Dirichlet problem, 
\[
\begin{cases} (-\Delta +q)v=qu & \text{in}\quad M^{\text{int}},\\
\gamma(v)=0,
\end{cases}
\]
has a unique solution $v\in (H^2\cap H^1_0)(M^{\text{int}})$.  Letting $w=u-v\in L^2(M)$, we see that $w\in b_q$ and $\gamma(w)=g$.   Note also that 
 \begin{equation}
\label{eq_ref_Eskin_2}
\|w\|_{L^2(M)}\le C\|g\|_{H^{-\frac{1}{2}}(\p M)}.
\end{equation}
To see the injectivity of $\gamma$, let $u\in b_q$ be such that $\gamma(u)=0$. Then $u\in (H^2\cap H^1_0)(M^{\text{int}})$, see \cite{Bukhgeim_Uhlmann_2002}, and therefore, $u=0$. This completes the proof of the claim.

Next we shall define suitable single layer operators associated to the Green operator $G_\varphi$. To that end, we first note that the trace map
\[
\gamma: e^{\frac{\pm \varphi}{h}}H_{\Delta}(M)\to e^{\frac{\pm \varphi}{h}}H^{-\frac{1}{2}}(\p M)=H^{-\frac{1}{2}}(\p M)
\]
is continuous. Thus, it follows from Theorem \ref{thm_green_operators} (iii) that the map
\[
\gamma\circ G_\varphi :L^2(M)\to H^{-\frac{1}{2}}(\p M)
\]
is continuous.  Thus, the $L^2$--adjoint 
\[
(\gamma\circ G_\varphi)^* :H^{\frac{1}{2}}(\p M)\to L^2(M)
\]
is also continuous.   Let $h\in H^{\frac{1}{2}}(\p M)$. We claim that 
\begin{equation}
\label{eq_2_9}
P_{-\varphi}((\gamma\circ G_\varphi)^* h)=0\quad \text{in}\quad \mathcal{D}'(M^{\text{int}}). 
\end{equation}
To see the claim let $f\in C_0^\infty(M^{\text{int}})$. Using Theorem \ref{thm_green_operators} (v), we get
\[
(P_{-\varphi}((\gamma\circ G_\varphi)^* h), f)_{L^2(M)}=(h,(\gamma\circ G_\varphi)P_\varphi f )_{H^{\frac{1}{2}}(\p M), H^{-\frac{1}{2}}(\p M)}=0,
\] 
showing \eqref{eq_2_9}. Now \eqref{eq_2_9} implies that for any $h\in H^{\frac{1}{2}}(\p M)$, 
$ (\gamma\circ G_\varphi)^* h\in e^{-\frac{\varphi}{h}} H_{\Delta}(M)$,
and therefore, the map
\[
(\gamma\circ G_\varphi)^*: H^{\frac{1}{2}}(\p M)\to e^{-\frac{\varphi}{h}} H_{\Delta}(M)
\]
is bounded. Thus, the map
\[
\gamma\circ (\gamma\circ G_\varphi)^*: H^{\frac{1}{2}}(\p M) \to H^{-\frac{1}{2}}(\p M)
\]
is well defined and bounded. Its $L^2$--adjoint 
\[
(\gamma\circ (\gamma\circ G_\varphi)^*)^*: H^{\frac{1}{2}}(\p M)\to H^{-\frac{1}{2}}(\p M)
\]
is bounded. 
We define the single layer operator by 
\begin{equation}
\label{eq_2_10}
S_\varphi= e^{-\frac{\varphi}{h}}(\gamma\circ (\gamma\circ G_\varphi)^*)^* e^{\frac{\varphi}{h}}\in \mathcal{L}(H^{\frac{1}{2}}(\p M), H^{-\frac{1}{2}}(\p M)). 
\end{equation}

Using the definition \eqref{eq_int_2_DN} of the Dirichlet--to--Neumann map, for $f,k\in H^{\frac{1}{2}}(\p M)$, we get
\begin{equation}
\label{eq_2_11}
\langle (\Lambda_{g,q}-\Lambda_{g,0})f, k\rangle_{H^{-\frac{1}{2}}(\p M), H^{\frac{1}{2}}(\p M)}=\int_M q u_1 u_2dV_g,
\end{equation}
where 
$u_1, u_2\in H^1(M^{\text{int}})$ are  such that 
\begin{equation}
\label{eq_3_2}
\begin{cases}
(-\Delta+q)u_1=0& \text{in} \quad M^{\text{int}},\\
u_1|_{\p M}=f, 
\end{cases}
\end{equation}
and 
\begin{equation}
\label{eq_3_3}
\begin{cases}
-\Delta u_2=0& \text{in} \quad M^{\text{int}},\\
u_2|_{\p M}=k.
\end{cases}
\end{equation}

We claim that $\Lambda_{g,q}-\Lambda_{g,0}$ extends to a linear continuous map $H^{-\frac{1}{2}}(\p M) \to H^{\frac{1}{2}}(\p M)$. Indeed, it follows  \eqref{eq_ref_Eskin_1} and \eqref{eq_ref_Eskin_2} that  
\begin{align*}
\big| \langle (\Lambda_{g,q}-\Lambda_{g,0})f, k\rangle_{H^{-\frac{1}{2}}(\p M), H^{\frac{1}{2}}(\p M)}\big|&\le C \|u_1\|_{L^2(M)} \| u_2\|_{L^2(M)}\\
&\le C\|f\|_{H^{-\frac{1}{2}}(\p M)}\|k\|_{H^{-\frac{1}{2}}(\p M)},
\end{align*}
and therefore, by density of $H^{\frac{1}{2}}(\p M)$ in $H^{-\frac{1}{2}}(\p M)$, we see that 
\[
\|(\Lambda_{g,q}-\Lambda_{g,0})f\|_{H^{\frac{1}{2}}(\p M)}\le C \|f\|_{H^{-\frac{1}{2}}(\p M)}.
\]
The claim follows.  Combining the claim with \eqref{eq_ref_Eskin_1},  \eqref{eq_ref_Eskin_2}, we see that the integral identity 
\eqref{eq_2_11} extends to all $f,k\in H^{-\frac{1}{2}}(\p M)$ with the corresponding solutions $u_1\in b_q$, $u_2\in b_0$,  and we obtain that 
\begin{equation}
\label{eq_2_11_new_ex}
\langle (\Lambda_{g,q}-\Lambda_{g,0})f, k\rangle_{H^{\frac{1}{2}}(\p M), H^{-\frac{1}{2}}(\p M)}=\int_M q u_1 u_2dV_g. 
\end{equation}

Now consider the map $S_\varphi(\Lambda_{g,q}-\Lambda_{g,0})\in \mathcal{L}(H^{-\frac{1}{2}}(\p M),H^{-\frac{1}{2}}(\p M))$.  Let 
\begin{equation}
\label{eq_Poisson}
\mathcal{P}_q=\gamma^{-1}: H^{-\frac{1}{2}}(\p M)\to b_q
\end{equation}
 be the Poisson operator. We claim that 
\begin{equation}
\label{eq_2_12}
S_\varphi(\Lambda_{g, q}-\Lambda_{g,0})=\gamma\circ e^{-\frac{\varphi}{h}}\circ G_\varphi\circ e^{\frac{\varphi}{h}} \circ q\circ \mathcal{P}_q
\end{equation}
in the sense of linear continuous maps:  $H^{-\frac{1}{2}}(\p M)\to H^{-\frac{1}{2}}(\p M)$. 
Indeed, let $f,k\in C^\infty(\p M)$ and note that in view of \eqref{eq_2_9}, the function $e^{\frac{\varphi}{h}}(\gamma\circ G_\varphi)^*e^{-\frac{\varphi}{h}}k\in L^2(M)$ is harmonic. Then using \eqref{eq_2_11_new_ex} and \eqref{eq_2_10},   we get 
\begin{align*}
(\gamma\circ e^{-\frac{\varphi}{h}}\circ &G_\varphi\circ e^{\frac{\varphi}{h}} \circ q\circ \mathcal{P}_qf, k)_{H^{-\frac{1}{2}}(\p M), H^{\frac{1}{2}}(\p M)}\\
&=
((\gamma\circ G_\varphi)\circ e^{\frac{\varphi}{h}} \circ q\circ \mathcal{P}_qf, e^{-\frac{\varphi}{h}} k)_{H^{-\frac{1}{2}}(\p M), H^{\frac{1}{2}}(\p M)}\\
&=( q\circ \mathcal{P}_qf, e^{\frac{\varphi}{h}}(\gamma\circ G_\varphi)^*e^{-\frac{\varphi}{h}} k)_{L^2(M)}\\
&=( (\Lambda_{g,q}-\Lambda_{g,0})f,  \gamma \circ e^{\frac{\varphi}{h}}(\gamma\circ G_\varphi)^*e^{-\frac{\varphi}{h}} k)_{H^{\frac{1}{2}}(\p M), H^{-\frac{1}{2}}(\p M)}\\
&=(S_\varphi (\Lambda_{g,q}-\Lambda_{g,0}) f, k )_{H^{-\frac{1}{2}}(\p M), H^{\frac{1}{2}}(\p M)}. 
\end{align*}
Thus, \eqref{eq_2_12} follows. 

\begin{prop}
\label{prop_integral_eq}
Let $k,f\in H^{-\frac{1}{2}}(\p M)$. Then 
\begin{equation}
\label{eq_2_13}
(1+ h^2S_\varphi(\Lambda_{g,q}-\Lambda_{g,0}))k=f
\end{equation}
if and only if 
\begin{equation}
\label{eq_2_14}
(1+ e^{-\frac{\varphi}{h}}\circ G_\varphi\circ e^{\frac{\varphi}{h}}  h^2q) \mathcal{P}_q k=\mathcal{P}_0 f. 
\end{equation}
\end{prop}

\begin{proof}
Assume first  that \eqref{eq_2_13} holds.  Applying $-h^2\Delta$ to the left hand side of \eqref{eq_2_14} and using Theorem \ref{thm_green_operators} (i), we get 
\[
(-h^2\Delta)(1+ e^{-\frac{\varphi}{h}}\circ G_\varphi\circ e^{\frac{\varphi}{h}}  h^2q) \mathcal{P}_q k=0. 
\]
Furthermore, using \eqref{eq_2_12} and \eqref{eq_2_13}, we see that 
\[
\gamma (1+ e^{-\frac{\varphi}{h}}\circ G_\varphi\circ e^{\frac{\varphi}{h}}  h^2q) \mathcal{P}_q k= k+h^2S_\varphi(\Lambda_{g,q}-\Lambda_{g,0})k=f.
\] 
Hence,  \eqref{eq_2_14}  follows. Assume now that  \eqref{eq_2_14} holds. Then taking trace in  \eqref{eq_2_14}, we obtain  \eqref{eq_2_13}.  
 \end{proof}

\begin{prop}
\label{prop_isomorphism}
The map $1+ h^2S_\varphi(\Lambda_{g,q}-\Lambda_{g,0}): H^{-\frac{1}{2}}(\p M)\to H^{-\frac{1}{2}}(\p M)$ is a linear homemorphism for all $0<h\ll 1$. 
\end{prop}

\begin{proof}
First by Theorem \ref{thm_green_operators} (ii), for all $0<h\ll 1$, the map
\[
1+ h^2e^{-\frac{\varphi}{h}}\circ G_\varphi\circ e^{\frac{\varphi}{h}}  q=e^{-\frac{\varphi}{h}}(1+h^2G_\varphi q )e^{\frac{\varphi}{h}}: L^2(M)\to L^2(M)
\]
is a linear homemorphism. Hence, for all $0<h\ll 1$ and any $v\in L^2(M)$, there exists a unique $u\in L^2(M)$ such that 
\[
(1+ h^2e^{-\frac{\varphi}{h}}\circ G_\varphi\circ e^{\frac{\varphi}{h}}  q)u=v. 
\]
If $v\in b_0$ then 
\[
0=-h^2\Delta u+ (-h^2\Delta)(e^{-\frac{\varphi}{h}}\circ G_\varphi\circ e^{\frac{\varphi}{h}}  h^2qu)= (-h^2\Delta+h^2q)u,
\]
and therefore, $u\in b_q$. Thus, for all $0<h\ll 1$, the map
\[
1+ h^2e^{-\frac{\varphi}{h}}\circ G_\varphi\circ e^{\frac{\varphi}{h}}  q: b_q\to b_0
\]
is an isomorphism. As $\mathcal{P}_q: H^{-\frac{1}{2}}(\p M)\to b_q$  is an 
 isomorphism, by Proposition \ref{prop_integral_eq}, we get  the claim of Proposition \ref{prop_isomorphism}.
\end{proof}

\section{Proof of Theorem \ref{thm_main_2}}

\label{sec_proof_of_thm_main}

Assume first that $(M,g)$ is transversally anisotropic, i.e. $c=1$, so that $g=e\oplus g_0$. Assume also  that $(M,g)$, and therefore $(M_0,g_0)$, are known, as is the Dirichlet--to--Neumann map $\Lambda_{g,q}$. We would like to provide a reconstruction procedure of $q$ from this data. 

Our starting point is the integral identity \eqref{eq_2_11_new_ex} valid for all  $u_1\in b_q$ and $u_2\in b_0$. 
We shall next construct $u_1$ and $u_2$ as special complex geometric optics solutions. To that end, we shall need the following result from \cite{DKurylevLS_2016} concerning existence of Gaussian beam quasimodes, concentrating along non-tangential geodesics in $M_0$. 

\begin{thm}[\cite{DKurylevLS_2016}] 
\label{thm_Mikko}
Let $(M_0,g_0)$ be a compact oriented manifold with smooth boundary, let $\gamma:[0,L]\to M_0$ be a non-tangential geodesic, and let $\lambda\in \R$. For any $K>0$, there is a family of functions $v_s\in C^\infty(M_0)$, where $s=\frac{1}{h}+i\lambda$, and $0<h\le 1$, such that 
\begin{equation}
\label{eq_3_4}
\|(-\Delta_{g_0}-s^2)v_s\|_{L^2(M_0)}=\mathcal{O}(h^K), \quad \|v_s\|_{L^2(M_0)}=\mathcal{O}(1), 
\end{equation}
as $h\to 0$, and for any $\psi\in C(M_0)$, one has 
\begin{equation}
\label{eq_3_5}
\lim_{h\to 0} \int_{M_0}|v_s|^2\psi dV_{g_0}=\int_0^L e^{-2\lambda t}\psi(\gamma(t))dt. 
\end{equation}
\end{thm}  

\textbf{Remark.} It follows from the proof of Theorem \ref{thm_Mikko} in \cite{DKurylevLS_2016} that $v_s$ are explicit functions of Gaussian type which can be constructed from the knowledge of the manifold $(M_0,g_0)$.  

Let us write $x=(x_1, x')$ for local coordinates in $\R\times M_0$. The function $\varphi(x)=x_1$ is a limiting Carleman weight for $-h^2\Delta$, see \cite{DKSaloU_2009}. It follows from \eqref{eq_3_4}, since $v_s=v_s(x')$, that 
\begin{equation}
\label{eq_3_6}
\|e^{-sx_1}(-\Delta_{g})e^{sx_1} v_s\|_{L^2(M_0)}=\mathcal{O}(h^K),\quad \|e^{sx_1}(-\Delta_{g})e^{-sx_1} v_s\|_{L^2(M_0)}=\mathcal{O}(h^K),
\end{equation}
as $h\to 0$.  

We are interested in harmonic functions $u_2\in L^2(M)$ of the form, 
\begin{equation}
\label{eq_3_7}
u_2=e^{sx_1}(v_s+\tilde r_2),
\end{equation}
where $v_s$ is the Gaussian beam quasimode from Theorem \ref{thm_Mikko}, and $\tilde r_2$ is a remainder term to be constructed. Now $u_2$ is harmonic if and only if $\tilde r_2$ satisfies
\begin{equation}
\label{eq_3_8}
P_{-\varphi} e^{i\lambda x_1}\tilde r_2= - e^{i\lambda x_1} e^{-sx_1}(-h^2\Delta)e^{sx_1} v_s.
\end{equation}
Looking for a solution of \eqref{eq_3_8} in the form $\tilde r_2= e^{-i\lambda x_1}G_{-\varphi}r_2$, we see that 
\[
r_2=-e^{i\lambda x_1}e^{-sx_1}(-h^2\Delta)e^{sx_1} v_s.
\]
It follows from \eqref{eq_3_6} that $\|r_2\|_{L^2(M)}=\mathcal{O}(h^{K+2})$, and therefore, by  Theorem \ref{thm_green_operators} (ii), 
\begin{equation}
\label{eq_3_8_2}
\|\tilde r_2\|_{L^2(M)}=\mathcal{O}(h^{K+1}),
\end{equation} 
as $h\to 0$, for $K>0$. 

We next construct complex geometric optics solutions $u_1\in b_q$ in the form, 
\begin{equation}
\label{eq_3_9}
u_1=u_0+e^{-sx_1}\tilde r_1,
\end{equation}
where $u_0\in L^2(M)$ is a harmonic function of the form,
\begin{equation}
\label{eq_3_10}
u_0=e^{-sx_1}(v_s+\tilde r_0).
\end{equation}
Here $\tilde r_0$ satisfies the equation
\begin{equation}
\label{eq_3_11}
P_{\varphi} e^{-i\lambda x_1}\tilde r_0= - e^{-i\lambda x_1} e^{sx_1}(-h^2\Delta)e^{-sx_1} v_s,
\end{equation}
and we can take
\[
\tilde r_0=e^{i\lambda x_1} G_\varphi r_0, \quad r_0=-e^{-i\lambda x_1} e^{sx_1}(-h^2\Delta)e^{-sx_1} v_s.
\]
It follows  from \eqref{eq_3_6} and  Theorem \ref{thm_green_operators} (ii) that 
\begin{equation}
\label{eq_3_11_2}
\|\tilde r_0\|_{L^2(M)}=\mathcal{O}(h^{K+1}), 
\end{equation}
as $h\to 0$, for $K>0$. 
To find the remainder $\tilde r_1$ in \eqref{eq_3_9}, we should solve the equation
\[
(-\Delta+q)(u_0+e^{-sx_1}\tilde r_1)=0,
\]
which is equivalent to 
\begin{equation}
\label{eq_3_12}
(P_\varphi+h^2q)e^{-i\lambda x_1}\tilde r_1=-h^2e^{\frac{\varphi}{h}}q u_0. 
\end{equation}
Looking for a solution $\tilde r_1$ of \eqref{eq_3_12} in the form 
\begin{equation}
\label{eq_3_12_2}
\tilde r_1=e^{i\lambda x_1}G_\varphi r_1,\quad r_1\in L^2(M), 
\end{equation}
 we see that $r_1$ should satisfy 
\begin{equation}
\label{eq_3_13}
(I+h^2qG_\varphi)r_1=-h^2e^{\frac{\varphi}{h}}q u_0. 
\end{equation}
It follows from Theorem \ref{thm_green_operators} (ii) that the equation \eqref{eq_3_13} has a unique solution $r_1\in L^2(M)$ such that 
\[
\|r_1\|_{L^2(M)}=\mathcal{O}(h^2)\|e^{\frac{\varphi}{h}} u_0\|_{L^2(M)}=\mathcal{O}(h^2)\|e^{-i\lambda x_1} (v_s+\tilde r_0)\|_{L^2(M)}=\mathcal{O}(h^2),
\]
as $h\to 0$. Here we have used 
\eqref{eq_3_10}, \eqref{eq_3_4}, and \eqref{eq_3_11_2}. Hence, 
\begin{equation}
\label{eq_3_13_2}
\|\tilde r_1\|_{L^2(M)}=\mathcal{O}(h),
\end{equation}
 as $h\to 0$.  We have therefore constructed $u_1\in b_q$  of the form 
 \begin{equation}
\label{eq_3_14}
u_1=u_0+e^{-\frac{\varphi}{h}}G_\varphi r_1,  
\end{equation}
where $r_1\in L^2(M)$ is the unique solution of \eqref{eq_3_13}.

 We shall next show that the boundary traces of $u_1$ can be reconstructed from the knowledge of $\Lambda_{g,q}$. To that end, we claim that $u_1$ satisfies 
\begin{equation}
\label{eq_3_15}
(I+h^2 e^{-\frac{\varphi}{h}}G_\varphi q e^{\frac{\varphi}{h}})u_1=u_0.
\end{equation}
Indeed, applying $G_\varphi$ to \eqref{eq_3_13} and multiplying by $e^{-\frac{\varphi}{h}}$, we get 
\begin{equation}
\label{eq_3_16}
e^{-\frac{\varphi}{h}} G_\varphi r_1+ e^{-\frac{\varphi}{h}} h^2 G_\varphi (q G_\varphi r_1)=-h^2e^{-\frac{\varphi}{h}} G_\varphi e^{\frac{\varphi}{h}} q u_0.
\end{equation}
Adding $u_0$ to the both sides of \eqref{eq_3_16}, and using \eqref{eq_3_14}, we obtain \eqref{eq_3_15} as claimed. 

Note  that the harmonic function  $u_0$ given by  \eqref{eq_3_10} is an explicit function which is constructed from the knowledge of $(M,g)$. In particular, $u_0|_{\p M}$ is also known. By Proposition \ref{prop_integral_eq} , $f=u_1|_{\p M}\in H^{-\frac{1}{2}}(\p M)$ satisfies the boundary integral equation,
\begin{equation}
\label{eq_3_17}
(1+ h^2S_\varphi(\Lambda_{g,q}-\Lambda_{g,0}))f=u_0|_{\p M}.
\end{equation}
Here the operator in the left hand side and the function in the right hand side are known, and 
by Proposition \ref{prop_isomorphism}, for all $0<h\ll 1$, we can construct $f$ as the unique solution to \eqref{eq_3_17}. 

It follows from the discussion above together with the  integral identity \eqref{eq_2_11_new_ex}, that from the knowledge of our data, we can reconstruct the  integrals 
\begin{equation}
\label{eq_3_18}
\int_M q u_1\overline{u_2}dV_g,
\end{equation}
with  $u_1$, $u_2$ given by \eqref{eq_3_9}, \eqref{eq_3_7}, respectively. Thus, using \eqref{eq_3_7}, \eqref{eq_3_9}, \eqref{eq_3_10}, we conclude from \eqref{eq_3_18} that we can reconstruct 
\begin{equation}
\label{eq_3_19}
\int_M q e^{-2i \lambda x_1}(v_s+\tilde r_0+\tilde r_1)(\overline{v_s}+\overline{\tilde r_2})dx_1dV_{g_0}, 
\end{equation}
for all $0<h\ll 1$. 
Using also \eqref{eq_3_4},  \eqref{eq_3_8_2}, \eqref{eq_3_11_2},  \eqref{eq_3_13_2}, we observe that \eqref{eq_3_19} is of the form 
\begin{equation}
\label{eq_3_20}
\int_M q e^{-2i \lambda x_1} |v_s|^2dx_1dV_{g_0}+\mathcal{O}(h),
\end{equation}
as $h\to 0$. It will now be convenient to extend the domain of integration in \eqref{eq_3_20} to all of $T=\R\times M_0$. To that end, we extend $q$ to a function in $C_0(T^{\text{int}})$ in such a way that  $q|_{T\setminus M}$ is known. This can be done by determining $q$ on $\p M$ from  the knowledge of $\Lambda_{g,q}$ and $\Lambda_{g,0}$ in a constructive  way, see Theorem \ref{thm_app_boundary_reconstruction}. Hence, it follows from \eqref{eq_3_20} that the boundary data allows us to reconstruct 
\begin{equation}
\label{eq_3_21}
\int_{\R} e^{-2i \lambda x_1} \int_{M_0} q(x_1, x')|v_s(x')|^2 dV_{g_0}dx_1 +\mathcal{O}(h).
\end{equation}
Taking the limit as $h\to 0$ in \eqref{eq_3_21} and using \eqref{eq_3_5}, we are able to reconstruct 
\begin{equation}
\label{eq_3_22}
\int_{\R} e^{-2i \lambda x_1} \int_{0}^L e^{-2\lambda t} q(x_1, \gamma(t))dt=\int_{0}^L \hat q(2\lambda, \gamma(t))e^{-2\lambda t}dt,
\end{equation}
for any $\lambda\in \R$ and any non-tangential geodesic $\gamma$ in $M_0$. Here 
\[
\hat q(\lambda, x')=\int_\R e^{-i \lambda x_1} q(x_1,x')dx_1.
\]
The integral in the right hand side of \eqref{eq_3_22} is the attenuated geodesic ray transform of $\hat q(2\lambda, \cdot)$ with constant attenuation $-2\lambda$. 

Setting $\lambda=0$ in \eqref{eq_3_22}, we recover the geodesic ray transform of $\hat q(0, \cdot)$. Using the constructive invertibility assumption for the  geodesic ray transform on $M_0$, we determine $\hat q(0, \cdot)$ in $M_0$. Differentiating  \eqref{eq_3_22} with respect to $\lambda$ and letting $\lambda=0$, as we know $\hat q(0, \cdot)$, we constructively determine the geodesic ray transform of $\p_\lambda \hat q(0, \cdot)$. Using again the constructive invertibility assumption, we constructively recover $\p_\lambda \hat q(0, \cdot)$ in $M_0$.
 Continuing in the same fashion, we constructively determine the derivatives  $\p_\lambda^k \hat q(0, \cdot)$ in $M_0$ for all $k\ge 0$. We have therefore determined the Taylor series of the entire function $\lambda \mapsto \hat q(\lambda, x')$ at $\lambda =0$. Inverting the one-dimensional Fourier transform, we complete the reconstruction of $q$ in $\R\times M_0$ in the case $c=1$. 

The argument explained in  \cite[Section 4]{Kenig_Salo_Uhlmann_2011} allows us to remove the simplifying assumption $c=1$. This completes the proof of Theorem \ref{thm_main_2}.

\begin{appendix}

\section{Boundary reconstruction for a continuous potential}

\label{app_boundary_reconstruction}

The purpose of this appendix is to provide a reconstruction formula for the boundary values of a continuous potential $q$ from the knowledge of the Dirichlet--to--Neumann map for the Schr\"odinger operator $-\Delta+q$ on a smooth compact Riemannian manifold of dimension $n\ge 2$ with smooth boundary. The boundary determination of a continuous potential   is known, see \cite[Appendix]{Guillarmou_Tzou_2011} for the case $n=2$, and \cite[Appendix C]{Krup_Uhlmann_2020} for an extension of this result  to the case $n\ge 3$, see also \cite{Aless_deHoop_Gaburro_Sincich_2018}.  We refer to \cite{Alessandrini_1990}, \cite{Brown_2001}, \cite{Brown_Salo_2006}, \cite{Kohn_Vogelius_1984},  \cite{Sylvester_Uhlmann_1988}, \cite{Caro_Andoni_2017}, \cite{Caro_Merono}, for the boundary determinations/reconstructions of conductivity as well as first order perturbations of the Laplacian.  The approach of \cite[Appendix]{Guillarmou_Tzou_2011} uses a family of functions, whose boundary values have a highly oscillatory behavior while becoming increasingly concentrated near a given point on the boundary of $M$, proposed in \cite{Brown_2001}, \cite{Brown_Salo_2006}, as well as Carleman estimates for the conjugated Laplacian with a gain of two derivatives in order to convert such functions into solutions of Schr\"odinger equations. This approach does not appear to be constructive, as the boundary traces of these solutions are not determined in this approach.  Not being aware of any reference for the constructive determination of the boundary values of a continuous potential from boundary measurements, and also, since we need this result for the proof of Theorem \ref{thm_main_2}, we present a reconstruction formula here. Let us remark that in the case of smooth potentials, the entire Taylor series at the boundary can be determined 
from the knowledge of the Dirichlet--to--Neumann map by means of a constructive procedure, see  \cite[Section 8]{DKSaloU_2009},  \cite{Stefanov_Yang_2018}, \cite{Lassas_Liimatainen_Salo}. 

Our boundary reconstruction result is as follows. 
\begin{thm}
\label{thm_app_boundary_reconstruction}
Let $(M,g)$ be a compact smooth Riemannian manifold of dimension $n\ge 2$ with smooth boundary. Let $q\in C(M)$ and assume that $0$ is not a Dirichlet eigenvalue of $-\Delta+q$ in $M$. For each point $x_0\in \p M$, there exists an explicit  family of functions $v_\lambda\in C^\infty(M)$, $0<\lambda\ll 1$,  such that 
\[
 q(x_0)= 2\lim_{\lambda\to 0} \langle (\Lambda_{g, q}- \Lambda_{g,0}) (v_\lambda|_{\p M}), \overline{v_\lambda}|_{\p M} \rangle_{H^{-\frac{1}{2}}(\p M), H^{\frac{1}{2}}(\p M)}.
\]
\end{thm}

\begin{proof}
Our starting point is the  integral identity \eqref{eq_2_11} which we write as follows, 
\begin{equation}
\label{eq_app_3}
\int_{M}q u\overline{v}dV_g=\langle (\Lambda_{g, q}- \Lambda_{g,0}) f, \overline{f} \rangle_{H^{-\frac{1}{2}}(\p M), H^{\frac{1}{2}}(\p M)},
\end{equation}
where $u,v\in H^1(M^{\text{int}})$ are solutions to 
\begin{equation}
\label{eq_app_1}
\begin{cases}
(-\Delta+q)u=0 & \text{in}\quad M^\text{int},\\
u|_{\p M}=f,
\end{cases}
\end{equation}
 and 
\begin{equation}
\label{eq_app_4}
\begin{cases}
-\Delta v=0 & \text{in}\quad M^\text{int},\\
v|_{\p M}=f.
\end{cases}
\end{equation}

Next  we shall follow \cite{Brown_2001}, \cite{Brown_Salo_2006}, constructing an explicit  family of functions $v_\lambda$, whose boundary values have a highly oscillatory behavior as $\lambda\to 0$, while becoming increasingly concentrated near a given point on the boundary of $M$.  To that end,  we let $x_0\in \p M$ and let $(x_1,\dots, x_n)$ be the boundary normal coordinates centered at $x_0$ so that in these coordinates, $x_0 =0$, the boundary $\p M$ is given by $\{x_n=0\}$, and $M^{\text{int}}$ is given by $\{x_n > 0\}$. We have, see \cite{Lee_Uhlmann_1989},   
\begin{equation}
\label{eq_boun_2_metric}
g(x',x_n)=\sum_{\alpha,\beta=1}^{n-1}g_{\alpha\beta}(x)dx_\alpha dx_\beta+(dx_n)^2,
\end{equation}
and we may also assume that the coordinates $x' = (x_1, \dots, x_{n-1})$ are chosen so that 
\begin{equation}
\label{eq_boun_2}
g^{\alpha\beta}(x',0)=\delta^{\alpha\beta}+\mathcal{O}(|x'|^2), \quad 1\le \alpha,\beta\le n-1,
\end{equation}
see \cite[Chapter 2, Section 8, p. 56]{Petersen_book}. 

Notice that in the local coordinates, $T_{x_0}\p M=\R^{n-1}$, equipped with the Euclidean metric. 
The unit tangent vector $\tau$ is then given by $\tau=(\tau',0)$ where $\tau'\in \R^{n-1}$, $|\tau'|=1$.   Associated to the tangent vector $\tau'$ is the covector $\xi'_\alpha=\sum_{\beta=1}^{n-1} g_{\alpha \beta}(0) \tau'_\beta=\tau'_\alpha\in T^*_{x_0}\p M$. 

Let $\eta\in C^\infty_0(\R^n;\R)$ be such that $\supp(\eta)$ is in a small neighborhood of $0$, and 
\begin{equation}
\label{eq_boun_4}
\int_{\R^{n-1}}\eta(x',0)^2dx'=1.
\end{equation}
Let $\frac{1}{3}\le \alpha\le \frac{1}{2}$. Following \cite{Brown_Salo_2006}, \cite[Appendix C]{Krup_Uhlmann_2020},  in the boundary normal coordinates,  we set 
\begin{equation}
\label{eq_boun_5}
v_\lambda(x)= \lambda^{-\frac{\alpha(n-1)}{2}-\frac{1}{2}}\eta\bigg(\frac{x}{\lambda^{\alpha}}\bigg)e^{\frac{i}{\lambda}(\tau'\cdot x'+ ix_n)}, \quad 0<\lambda\ll 1,
\end{equation}
so that  $v_\lambda\in C^\infty(M)$, with $\supp(v_\lambda)$ in  $\mathcal{O}(\lambda^{\alpha})$ neighborhood of $x_0=0$. Here $\tau'$ is viewed as a covector.  A direct computation shows that 
\begin{equation}
\label{eq_boun_6}
\|v_\lambda\|_{L^2(M)}=\mathcal{O}(1),  
\end{equation}
as $\lambda\to 0$, see also \cite[Appendix C]{Krup_Uhlmann_2020}.

Next we would like to construct $v, u\in H^1(M^{\text{int}})$  of the form 
\begin{equation}
\label{eq_boun_6_00}
v=v_\lambda+r_1,\quad u=v_\lambda+r_2,
\end{equation}
solving  \eqref{eq_app_4} and \eqref{eq_app_1}, and so that  $r_1$ and $r_2$ have decaying $L^2$ norms as $\lambda\to 0$.  The idea of  \cite[Appendix]{Guillarmou_Tzou_2011}, see also \cite[Appendix C]{Krup_Uhlmann_2020}, was to use Carleman estimates with a gain of two derivatives to accomplish this.  However, here for the reconstruction purposes we need to know the boundary traces $v|_{\p M}$ and $u|_{\p M}$, see \eqref{eq_app_3}. To achieve this, following \cite{Brown_2001}, \cite{Brown_Salo_2006}, see also \cite[Appendix]{Krup_Uhlmann_2018}, we shall obtain $r_1\in H^1_0(M^{\text{int}})$  as the solution to the Dirichlet problem, 
\begin{equation}
\label{eq_boun_34}
\begin{cases}
-\Delta r_1=\Delta  v_\lambda & \text{in}\quad M^{\text{int}},\\
r_1|_{\p M}=0. 
\end{cases}
\end{equation}
Similarly, we shall find $r_2\in H^1_0(M^{\text{int}})$  as the solution to the Dirichlet problem, 
\begin{equation}
\label{eq_boun_34_with_q}
\begin{cases}
(-\Delta +q)r_2=-(-\Delta v_\lambda+qv_\lambda) & \text{in}\quad M^{\text{int}},\\
r_2|_{\p M}=0. 
\end{cases}
\end{equation}
Note that in \cite{Brown_Salo_2006}, see also \cite[Appendix]{Krup_Uhlmann_2018}, one shows that  $\|r_1\|_{L^2(M)}=\mathcal{O}(1)$, which is not enough for the determination of the potential $q$ on $\p M$. To get an improved bound for $\|r_1\|_{L^2(M)}$, we use the estimate of Theorem \ref{thm_reg_1} below with $s=\frac{1}{2}+\varepsilon$, $0<\varepsilon<1/2$ to be chosen fixed. We have
\begin{equation}
\label{eq_app_7}
\|r_1\|_{L^2(M)}\le \|r_1\|_{H^{1/2+\varepsilon}(M^{\text{int}})}\le C\|\Delta v_\lambda\|_{H^{-3/2+\varepsilon}(M^{\text{int}})}.
\end{equation}
To estimate $\|\Delta v_\lambda\|_{H^{-3/2+\varepsilon}(M^{\text{int}})}$ we use interpolation, 
\begin{equation}
\label{eq_app_8}
\|\Delta v_\lambda\|_{H^{-3/2+\varepsilon}(M^{\text{int}})}\le \|\Delta v_\lambda\|_{H^{-1}(M^{\text{int}})}^{1/2+\varepsilon}\|\Delta v_\lambda\|_{H^{-2}(M^{\text{int}})}^{1/2-\varepsilon},
\end{equation}
see \cite[Theorem 7.22, p. 189]{Grubb_book}. Using the following bounds, established in \cite[Appendix C]{Krup_Uhlmann_2020},  
\[
\|\Delta v_\lambda\|_{H^{-2}(M^{\text{int}})}=\mathcal{O}(\lambda^{-\alpha+1}), \quad \frac{1}{3}\le \alpha\le \frac{1}{2},
\]
\[
\|\Delta v_\lambda\|_{H^{-1}(M^{\text{int}})}=\mathcal{O}(\lambda^{-\alpha}), \quad \frac{1}{3}\le \alpha\le \frac{1}{2},
\]
we get from \eqref{eq_app_8} that 
\[
\|\Delta v_\lambda\|_{H^{-3/2+\varepsilon}(M^{\text{int}})}=\mathcal{O}(\lambda^{-\alpha+1/2-\varepsilon}).
\]
Choosing $\alpha=1/3$ and $\varepsilon=1/12$, see see that 
\begin{equation}
\label{eq_app_9_0}
\|\Delta v_\lambda\|_{H^{-3/2+\varepsilon}(M^{\text{int}})}=\mathcal{O}(\lambda^{1/12}).
\end{equation}

Therefore, it follows from \eqref{eq_app_7} and \eqref{eq_app_9_0} that 
\begin{equation}
\label{eq_app_9}
\|r_1\|_{L^2(M)}=\mathcal{O}(\lambda^{1/12}),
\end{equation}
as $\lambda\to 0$.

In view of \eqref{eq_boun_34_with_q}, using the estimate of Theorem \ref{thm_reg_2} below with $s=\frac{1}{2}+\varepsilon$, $\varepsilon=1/12$, we obtain that 
\begin{equation}
\label{eq_app_10}
\|r_2\|_{L^2(M)}\le \|r_2\|_{H^{1/2+\varepsilon}(M^{\text{int}})}\le C(\|\Delta v_\lambda\|_{H^{-3/2+\varepsilon}(M^{\text{int}})}+\|qv_\lambda\|_{H^{-1}(M^{\text{int}})}).
\end{equation}
To bound $\|qv_\lambda\|_{H^{-1}(M^{\text{int}})}$ we first note that as $q\in C(M)$, using a partition of unity argument together with a regularization in each coordinate patch, we get that there exists $q_\tau\in C^\infty_0(M^{\text{int}})$, $\tau>0$, such that 
\begin{equation}
\label{eq_app_11}
\|q-q_\tau\|_{L^\infty(M)}=o(1), \quad  \|q_\tau\|_{L^\infty(M)}=\mathcal{O}(1), \quad  \|\nabla q_\tau\|_{L^\infty(M)}=\mathcal{O}(\tau^{-1}),
\end{equation}
as $\tau\to 0$.   Letting $\psi\in C_0^\infty(M^{\text{int}})$ and using \eqref{eq_app_11}, we obtain that 
\begin{equation}
\label{eq_app_12}
\bigg|\int_M (q-q_\tau)v_\lambda \psi dV_g\bigg|\le o_{\tau\to 0}(1)\mathcal{O}_{\lambda\to 0}(1)\|\psi\|_{L^2(M)}.
\end{equation}

Setting 
\[
L=\frac{\nabla \overline{\phi}\cdot \nabla}{i |\nabla \phi |^2}=\frac{1}{2i}\nabla \overline{\phi}\cdot \nabla, \quad \phi=\tau'\cdot x'+ix_n,
\]
we have $Le^{\frac{i}{\lambda}(\tau'\cdot x'+ix_n)}=\lambda^{-1}e^{\frac{i}{\lambda}(\tau'\cdot x'+ix_n)}$. Using \eqref{eq_boun_5}, the fact that the transpose $L^t= -L$ and integrating by parts, we get 
\begin{equation}
\label{eq_app_13}
\begin{aligned}
\bigg|\int_M q_\tau v_\lambda \psi dV_g\bigg|=\lambda^{-\frac{\alpha(n-1)}{2}-\frac{1}{2}}\bigg|\int_M q_\tau  \psi  \eta\bigg(\frac{x}{\lambda^{\alpha}}\bigg) e^{\frac{i}{\lambda}(\tau'\cdot x'+ ix_n)}|g(x)|^{1/2} dx\bigg|\\
=\lambda\lambda^{-\frac{\alpha(n-1)}{2}-\frac{1}{2}}\bigg|\int_M L\bigg( q_\tau  \psi  \eta\bigg(\frac{x}{\lambda^{\alpha}}\bigg)|g(x)|^{1/2}\bigg) e^{\frac{i}{\lambda}(\tau'\cdot x'+ ix_n)} dx\bigg|.
\end{aligned}
\end{equation}
Note that the worst growth in $\lambda$ in \eqref{eq_app_13}  occurs when $L$ falls on $q_\tau$ or $\eta$. Choosing $\tau=\lambda^{\alpha}$, and using the Cauchy--Schwarz inequality, those two terms are bounded by 
$\mathcal{O}(\lambda^{1-\alpha})\|\psi\|_{L^2(M)}$. Hence, it follows from \eqref{eq_app_12} and \eqref{eq_app_13} that 
\begin{equation}
\label{eq_app_14}
\|qv_\lambda\|_{H^{-1}(M^{\text{int}})}=o(1),
\end{equation}
as $\lambda\to 0$.  Thus, we see from \eqref{eq_app_10}, \eqref{eq_app_9_0}, and \eqref{eq_app_14}, that 
\begin{equation}
\label{eq_app_15}
\|r_2\|_{L^2(M)}=o(1),
\end{equation}
as $\lambda\to 0$.

Substituting $u$ and $v$ given by \eqref{eq_boun_6_00} into the integral identity \eqref{eq_app_3}, and taking the limit $\lambda\to 0$, we get 
\begin{equation}
\label{eq_app_17}
\lim_{\lambda\to 0}\langle (\Lambda_{g, q}- \Lambda_{g,0}) (v_\lambda|_{\p M}), \overline{v_\lambda}|_{\p M} \rangle_{H^{-\frac{1}{2}}(\p M), H^{\frac{1}{2}}(\p M)}=\lim_{\lambda\to 0} (I_1+I_2),
\end{equation}
where 
\[
I_1=\int_{M}q |v_\lambda|^2dV_g, \quad I_2=\int_{M}q (v_\lambda\overline{r_1}+ \overline{v_\lambda}r_2+\overline{r_1}r_2)dV_g.
\]
In view of \eqref{eq_app_9},  \eqref{eq_app_15}, and \eqref{eq_boun_6}, we have 
\begin{equation}
\label{eq_app_18}
|I_2|=o(1),
\end{equation}
as $\lambda \to 0$. Using \eqref{eq_boun_5}, \eqref{eq_boun_4}, the fact that $q$ is continuous, and making  the change of variables $y'=\frac{x'}{\lambda^{\alpha}}$, $y_n=\frac{x_n}{\lambda}$, we get 
\begin{equation}
\label{eq_boun_25}
\begin{aligned}
\lim_{\lambda\to 0}I_1= \lim_{\lambda\to 0}\int_{\R^{n-1}}\int_0^\infty q(\lambda^{\alpha}y',\lambda y_n)\eta^2(y', \lambda^{1-\alpha}y_n)e^{-2y_n} |g(\lambda^{\alpha}y', \lambda y_n)|^{1/2}dy'dy_n\\
=q(0)|g(0)|^{1/2}\int_0^{+\infty} e^{-2y_n}dy_n=\frac{1}{2}q(0).
\end{aligned}
\end{equation}
It follows from 
\eqref{eq_app_17}, \eqref{eq_app_18}, \eqref{eq_boun_25} that 
\[
\lim_{\lambda\to 0}\langle (\Lambda_{g, q}- \Lambda_{g,0}) (v_\lambda|_{\p M}), \overline{v_\lambda}|_{\p M} \rangle_{H^{-\frac{1}{2}}(\p M), H^{\frac{1}{2}}(\p M)}=\frac{1}{2}q(0)
\]
This completes the proof. 
\end{proof}

In the course of the proof of Theorem \ref{thm_app_boundary_reconstruction} we need the following result, see \cite[Section 54.2]{Eskin_book}.
\begin{thm}
\label{thm_reg_1}
Let $(M,g)$ be a compact smooth Riemannian manifold of dimension $n\ge 2$ with smooth boundary. Let $s>1/2$, $F\in H^{s-2}(M^{\emph{\text{int}}})$, $f\in H^{s-1/2}(\p M)$. Then the  Dirichlet problem
\[
\begin{cases}
-\Delta u=F & \text{in}\quad M^{\emph{\text{int}}},\\
u|_{\p M}=f,
\end{cases}
\]
has a unique solution $u\in H^s(M^{\emph{\text{int}}})$ and moreover, 
\[
\|u\|_{H^s(M^{\emph{\text{int}}})}\le C(\|F\|_{H^{s-2}(M^{\emph{\text{int}}})}+\|f\|_{H^{s-1/2}(\p M)}). 
\]
\end{thm}

We also need a similar result for the Dirichlet problem for the Schr\"odinger equation. 
\begin{thm}
\label{thm_reg_2}
Let $(M,g)$ be a compact smooth Riemannian manifold of dimension $n\ge 2$ with smooth boundary, and $q\in L^\infty(M)$. Assume that $0$ is not a Dirichlet eigenvalue of $-\Delta_g+q$. Let $1/2<s<2$, $F\in H^{s-2}(M^{\emph{\text{int}}})$, $f\in H^{s-1/2}(\p M)$. Then  the Dirichlet problem
\[
\begin{cases}
(-\Delta +q) u=F & \text{in}\quad M^{\emph{\text{int}}},\\
u|_{\p M}=f,
\end{cases}
\]
has a unique solution $u\in H^s(M^{\emph{\text{int}}})$ and moreover, 
\[
\|u\|_{H^s(M^{\emph{\text{int}}})}\le C(\|F\|_{H^{s-2}(M^{\emph{\text{int}}})}+\|f\|_{H^{s-1/2}(\p M)}). 
\]
\end{thm}

\begin{proof}

Consider the operators 
\[
\mathcal{A}: H^s(M^{\text{int}})\to H^{s-2}(M^{\text{int}})\times H^{s-1/2}(\p M),\quad u\mapsto ((-\Delta +q) u, u|_{\p M}),
\]
\[
\mathcal{A}_0: H^s(M^{\text{int}})\to H^{s-2}(M^{\text{int}})\times H^{s-1/2}(\p M),\quad u\mapsto (-\Delta u, u|_{\p M}),
\]
and 
\[
Q: H^s(M^{\text{int}})\to H^{s-2}(M^{\text{int}})\times H^{s-1/2}(\p M),\quad u\mapsto (q u, 0). 
\]
It follows from \cite[Section 54.2]{Eskin_book}, cf. Theorem \ref{thm_reg_1}, that $\mathcal{A}_0$ is  an isomorphism. The operator $Q$ is compact, as the operator $H^s(M^{\text{int}}) \ni u\mapsto q u\in  H^{s-2}(M^{\text{int}})$ is compact. The later follows from the fact that the operator $H^s(M^{\text{int}}) \ni u\mapsto q u\in L^2(M)$ is continuous and the embedding $L^2(M)\subset H^{s-2}(M^{\text{int}})$ is compact provided $s<2$. Hence, $\mathcal{A}=\mathcal{A}_0+Q$ is  Fredholm of index zero,  and as $0$ is not a Dirichlet eigenvalue of $-\Delta_g+q$, $\mathcal{A}$ is an isomorphism. 
\end{proof}

\end{appendix}

\section*{Acknowledgements}

A.F. was supported by EPSRC grant EP/P01593X/1.
K.K. is very grateful to Fran\c{c}ois Monard for useful discussions and references. K.K. would also like to thank Giovanni Alessandrini for the reference \cite{Aless_deHoop_Gaburro_Sincich_2018}. The research of K.K. is partially supported by the National Science Foundation (DMS 1815922). L.O. was supported by EPSRC grants EP/R002207/1 and EP/P01593X/1. The research of G.U. is partially supported by NSF, a Walker Professorship at UW and a Si-Yuan Professorship at IAS, HKUST. Part of the work was supported by the NSF grant DMS-1440140 while K.K. and G.U.  were in residence at MSRI in Berkeley, California, during Fall 2019 semester.

\end{document}